\documentclass[12pt,reqno]{amsart}

\usepackage{amsmath,amssymb,amsthm}
\usepackage[margin=1in]{geometry}
\usepackage{xcolor,tikz}
\usetikzlibrary{arrows.meta,calc}
\usepackage{hyperref}
\hypersetup{colorlinks=true, linkcolor=blue!50!black, citecolor=blue!50!black,
  urlcolor=blue!50!black}

\definecolor{coli}{HTML}{1F6FC4}   
\definecolor{colj}{HTML}{C07A12}   

\theoremstyle{plain}
\newtheorem{theorem}{Theorem}[section]
\newtheorem{proposition}[theorem]{Proposition}
\newtheorem{lemma}[theorem]{Lemma}

\theoremstyle{definition}
\newtheorem{definition}[theorem]{Definition}
\newtheorem{remark}[theorem]{Remark}

\newtheorem{declaration}[theorem]{Declaration}

\newcommand{\NN}{\mathbb{N}}

\newcommand{\Sk}{\operatorname{Sk}}
\newcommand{\eps}{\varepsilon}
\newcommand{\Em}{\mathsf{E}}
\newcommand{\Rec}{\mathsf{R}}
\DeclareMathOperator{\rge}{r}
\DeclareMathOperator{\src}{s}

\begin{document}

\title[Planar higher-rank trees have rank at most four]
{Planar non-degenerate higher-rank trees have rank at most four}

\author{David Pask}
\address{Math Department, James Cook University, 1 James Cook Dr, Douglas QLD 4814, Australia}
\email{david.a.pask@gmail.com}
\date{\today}

\begin{abstract}
We prove that a finite, connected, singly connected, locally convex higher-rank
tree whose $1$-skeleton is planar and which is \emph{non-degenerate}, in the sense
that every edge of each colour forms a commuting square with every other colour,
has rank at most four. Under these hypotheses this establishes the planarity
conjecture stated in \cite{Pask}. The obstruction side of the argument uses only the
non-planarity of $K_5$; it makes no appeal to the four-colour theorem. The engine
is a monotonicity property of the set of colours emitted at a vertex
(``backward propagation''), which forces, in any finite singly connected
non-degenerate $k$-graph, a single vertex emitting all $k$ colours; once $k\ge 5$,
local convexity manufactures a subdivision of $K_5$ at such a vertex.
\end{abstract}


\maketitle

\section{Introduction}

Higher-rank graphs ($k$-graphs) were introduced by Kumjian and Pask \cite{KP} as a
combinatorial model for a class of $C^\ast$-algebras. A rank-one graph is an
ordinary directed graph, and its path category is free; for $k\ge 2$ the category
is genuinely two-dimensional, encoded by a $k$-coloured $1$-skeleton together with
a complete collection of bi-coloured commuting squares \cite{HRSW,RSY}. Recent
work has studied $k$-graphs from a combinatorial and topological standpoint, in
particular their fundamental group and groupoid \cite{PRQ,KPW,KPSS}, and the class
of \emph{higher-rank trees}: connected $k$-graphs with trivial fundamental group.

In \cite{Pask} a family of higher-rank trees is constructed from polyhedral graphs.
Each member is connected, singly connected, locally convex, acyclic, embeds in its
fundamental groupoid, and is \emph{planar}: its $1$-skeleton $\Sk$ is a planar
graph, satisfying the higher-rank analogue
$\lvert\Sk^1\rvert-2\lvert\Sk^0\rvert+4=0$ of the elementary identity
$\lvert T^1\rvert-\lvert T^0\rvert+1=0$ for a finite rank-one tree. The
construction produces ranks $2,3,4$ only; the bound $4$ enters through the
four-colour theorem applied to the faces of the polyhedron. In
\cite[\S5, Question]{Pask} we asked for intrinsic criteria guaranteeing planarity, noting
that \emph{degenerate} examples exist in arbitrarily high rank --- examples in
which ``not all edges of each degree have relations between them'' --- and
conjecture a theorem ``akin to the four-colour theorem'' bounding the rank of a
planar higher-rank tree.

The purpose of this note is to prove such a bound. We isolate the non-degeneracy
condition implicit in the wording of \cite{Pask} --- that every edge of each colour forms a commuting square with
every other colour, which we call \emph{edge-level non-degeneracy}
(Definition~\ref{def:nd}) --- and prove the following.

\begin{theorem}\label{thm:main}
Let $\Lambda$ be a finite, connected, singly connected, locally convex higher-rank
tree of rank $k\ge 5$. If $\Lambda$ is edge-level non-degenerate, then its
$1$-skeleton $\Sk_\Lambda$ is non-planar. Consequently, a finite, connected, singly
connected, locally convex, edge-level non-degenerate higher-rank tree with planar
$1$-skeleton has rank at most $4$.
\end{theorem}

The proof has two ingredients, both elementary. First, the colours emitted at a
vertex satisfy a monotonicity law: along any edge the edges emitted by the range vertex
contains that of the source vertex (Lemma~\ref{lem:prop}). Combined with
edge-level non-degeneracy this forces, in a finite singly connected $k$-graph, a
single vertex that emits all $k$ colours (Proposition~\ref{prop:conc}); the finite
acyclicity supplied by single connectedness completes the argument. Second,
at a vertex emitting five colours, local convexity closes all ten of their
pairwise squares, and single connectedness makes the resulting ten ``diagonal''
vertices distinct, producing a subdivision of $K_5$ in the $1$-skeleton
(Lemma~\ref{lem:k5}). Non-planarity then follows from Kuratowski's theorem
\cite{Kur,Diestel}.

We stress that the obstruction uses only that $K_5$ is non-planar; the full
four-colour theorem \cite{AppelHaken} is required for the \emph{achievability} of
rank $4$ in the construction of \cite{Pask}, not for the bound proved here. The hypotheses and
their role are discussed in Section~\ref{sec:remarks}.

\section{Preliminaries}\label{sec:prelim}

We follow the conventions of \cite{KP,RSY,HRSW}. Let $\NN^k$ be the free abelian
monoid with generators $\eps_1,\dots,\eps_k$, ordered coordinatewise.

\begin{definition}[$k$-graph]
A \emph{higher-rank graph of rank $k$} (or \emph{$k$-graph}) is a countable category
$\Lambda$ together with a functor $d\colon\Lambda\to\NN^k$ satisfying the
\emph{factorisation property}: whenever $d(\lambda)=m+n$ there exist unique
$\mu,\nu\in\Lambda$ with $d(\mu)=m$, $d(\nu)=n$ and $\lambda=\mu\nu$. We write
$\rge,\src\colon\Lambda\to\Lambda^0$ for the range and source maps and
$\Lambda^m:=d^{-1}(m)$. Composition $\mu\nu$ is defined when $\src(\mu)=\rge(\nu)$,
and then $\rge(\mu\nu)=\rge(\mu)$, $\src(\mu\nu)=\src(\nu)$.
\end{definition}

\noindent
Since a $k$-graph is a category, elements are composed from right to left. For this reason we often draw edges pointing from right to left (cf.\ Definition~\ref{def:square} below.

Elements of $\Lambda^0$ are \emph{vertices}; elements of $\Lambda^{\eps_i}$ are
\emph{edges of colour $i$}. For $v,w\in\Lambda^0$ and $F\subseteq\Lambda$ we write,
following \cite{RSY}, $vF=\rge^{-1}(v)\cap F$ and $Fw=\src^{-1}(w)\cap F$. The
\emph{$1$-skeleton} $\Sk_\Lambda$ is the $k$-coloured directed graph with vertices
$\Lambda^0$ and edges $\bigcup_{i\le k}\Lambda^{\eps_i}$, an edge $f$ having colour
$i$ if and only if $f\in\Lambda^{\eps_i}$.

\begin{definition}[Commuting squares]\label{def:square}
For $i\ne j$, a \emph{commuting $(i,j)$-square} is an element $\lambda\in
\Lambda^{\eps_i+\eps_j}$, presented through its two factorisations
\[
\begin{tikzpicture}

\node[inner sep=0.8pt] (00) at (0,-1) {$\bullet$};
\node[inner sep=0.8pt] (10) at (1,-1) {$\bullet$}
	edge[->,blue] node[auto,black] {$\nu'$} (00);
\node[inner sep=0.8pt] (01) at (0,0) {$\bullet$}
	edge[->,red,dashed] node[auto,black,swap] {$\mu$} (00);
\node[inner sep=0.8pt] (11) at (1,0) {$\bullet$}
	edge[->,blue] node[auto,black,swap] {$\nu$} (01)
	edge[->,red,dashed] node[auto,black] {$\mu'$} (10);

\node[inner sep=8pt] at (-7,-0.5) {$\lambda=\mu\nu=\nu'\mu',\qquad
d(\mu)=d(\mu')=\eps_i,\quad d(\nu)=d(\nu')=\eps_j . $};
\end{tikzpicture}
\]
Its four \emph{sides} are the edges $\mu,\nu,\nu',\mu'$. An edge $e$
\emph{lies in} (or \emph{borders}) the square if $e\in\{\mu,\nu,\nu',\mu'\}$.
\end{definition}

By the factorisation property, $(i,j)$-squares are in bijection with
$\Lambda^{\eps_i+\eps_j}$, and each square determines a unique cofactorisation;
this is the completeness of the canonical set of squares of a $k$-graph
\cite[Lemma 4.2]{HRSW}.

\begin{definition}[Local convexity {\cite[Definition~3.9]{RSY}}]\label{def:lc}
$\Lambda$ is \emph{locally convex} if for all $i\ne j$ and all $e\in
\Lambda^{\eps_i}$, $f\in\Lambda^{\eps_j}$ with $\rge(e)=\rge(f)$, there exist
$e'\in\src(f)\Lambda^{\eps_i}$ and $f'\in\src(e)\Lambda^{\eps_j}$ with
$ef'=fe'$.
\end{definition}

\noindent
That is, two differently coloured edges with a common range close up to a
commuting square.

\begin{definition}[Connected, singly connected {\cite[Defnition~2.6]{Pask}}]
$\Lambda$ is \emph{connected} if the equivalence relation on $\Lambda^0$ generated
by $\{(u,v):u\Lambda v\ne\emptyset\}$ is all of $\Lambda^0\times\Lambda^0$, and
\emph{singly connected} if $\lvert u\Lambda v\rvert\le 1$ for all $u,v$.
\end{definition}

\begin{definition}[Higher-rank tree]
A \emph{higher-rank tree} is a connected $k$-graph with trivial fundamental group
(in the sense of \cite{PRQ}). The trees of \cite[Theorem~A]{Pask} are moreover singly
connected and locally convex; these are the only structural properties we use
below.
\end{definition}

We call $\Lambda$ \emph{finite} if $\Lambda^0$ is finite. The $1$-skeleton is
\emph{planar} if the underlying undirected graph of $\Sk_\Lambda$ admits a planar
embedding; this is the meaning of ``planar'' in \cite[Theorem~A]{Pask}. The only graph with one vertex which is a tree consists of just that vertex (and is planar), so we shall assume $\Lambda$ has more that one vertex.

The only fact we use about single connectedness is the absence of nontrivial
directed loops.

\begin{lemma}[No directed loops]\label{lem:noloop}
If $\Lambda$ is singly connected and $\lambda\in v\Lambda v$ then $d(\lambda)=0$,
i.e. $\lambda=\mathrm{id}_v$.
\end{lemma}

\begin{proof}
The identity $\mathrm{id}_v\in v\Lambda v$ has $d(\mathrm{id}_v)=0$. Any
$\lambda\in v\Lambda v$ with $d(\lambda)\ne 0$ is a second element, contradicting
$\lvert v\Lambda v\rvert\le 1$.
\end{proof}

\section{Emission, reception, and propagation}

\noindent
Throughout this section $\Lambda$ is a $k$-graph.

\begin{definition}\label{def:ER}
For $v\in\Lambda^0$ set
\[
\Em(v)=\{\,i: v\Lambda^{\eps_i}\ne\emptyset\,\},
\qquad
\Rec(v)=\{\,i: \Lambda^{\eps_i}v\ne\emptyset\,\}.
\]
We say $v$ \emph{emits} colour $i$ if $i\in\Em(v)$ (there is a colour-$i$ edge with
range $v$) and \emph{receives} colour $i$ if $i\in\Rec(v)$ (there is a colour-$i$
edge with source $v$).
\end{definition}

\begin{definition}[Edge-level non-degeneracy]\label{def:nd}
$\Lambda$ is \emph{edge-level non-degenerate} if for every colour $i$, every
$e\in\Lambda^{\eps_i}$, and every $j\ne i$, the edge $e$ lies in some commuting
$(i,j)$-square.
\end{definition}

\noindent
This is the condition that every edge of each colour ``has relations with'' every
other colour; its negation is the degeneracy noted in \cite[\S5]{Pask}.

\begin{lemma}[Boxing criterion]\label{lem:count}
Let $e\in\Lambda^{\eps_i}$ and $j\ne i$. Then $e$ lies in an $(i,j)$-square if and
only if
\[
j\in\Em(\src(e))\cup\Rec(\rge(e)).
\]
Consequently $\Lambda$ is edge-level non-degenerate if and only if
\[
\Em(\src(e))\cup\Rec(\rge(e))\;\supseteq\;\{1,\dots,k\}\setminus\{i\}
\qquad\text{for every }e\in\Lambda^{\eps_i}.
\]
\end{lemma}

\begin{proof}
By Definition~\ref{def:square}, $e$ is a colour-$i$ side of an $(i,j)$-square
$\lambda=\mu\nu=\nu'\mu'$ if and only if $e=\mu$ or $e=\mu'$.

If $e=\mu$ then $\lambda=e\nu$ with $\nu\in\Lambda^{\eps_j}$ and
$\src(e)=\rge(\nu)$, so $\nu\in\src(e)\Lambda^{\eps_j}$; such a $\lambda$ exists if and only if
$\src(e)\Lambda^{\eps_j}\ne\emptyset$, i.e. $j\in\Em(\src(e))$.

If $e=\mu'$ then $\lambda=\nu'e$ with $\nu'\in\Lambda^{\eps_j}$ and
$\src(\nu')=\rge(e)$, so $\nu'\in\Lambda^{\eps_j}\rge(e)$; such a $\lambda$ exists
if and only if $\Lambda^{\eps_j}\rge(e)\ne\emptyset$, i.e. $j\in\Rec(\rge(e))$.

Hence $e$ lies in an $(i,j)$-square if and only if $j\in\Em(\src(e))\cup\Rec(\rge(e))$. The
second statement is the conjunction over $j\ne i$.
\end{proof}

\begin{lemma}[Backward propagation]\label{lem:prop}
For every $e\in\Lambda^{\eps_i}$,
\[
\Em(\rge(e))\;\supseteq\;\{i\}\cup\Em(\src(e)).
\]
\noindent
For instance:
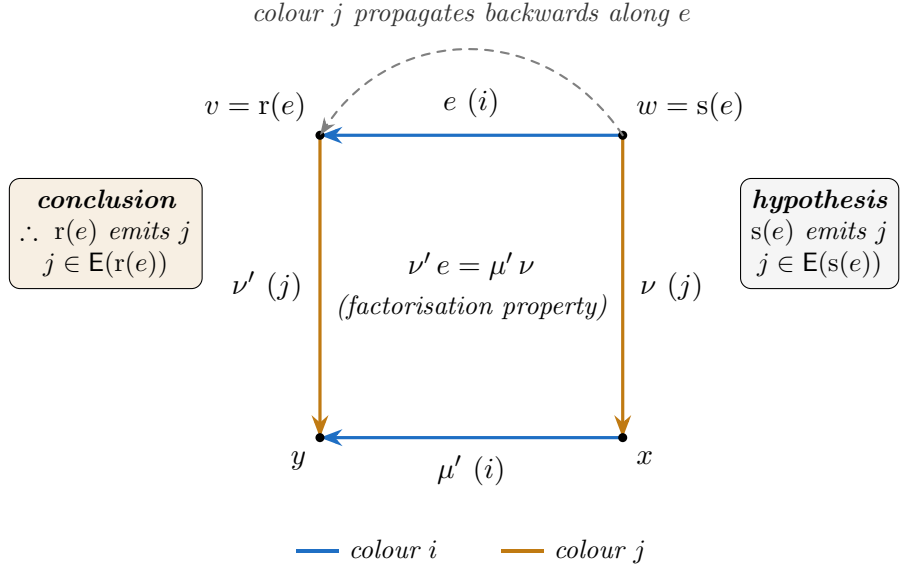
\[
\begin{tikzpicture}[
    >=Stealth, line cap=round,
    every node/.style={font=\small},
    edgei/.style={coli, very thick, ->},
    edgej/.style={colj, very thick, ->}
  ]
 
  \coordinate (v) at (0,4);
  \coordinate (w) at (4,4);
  \coordinate (y) at (0,0);
  \coordinate (x) at (4,0);
 
  \draw[edgei] (w) -- (v) node[midway, above=2pt, black] {$e\ (i)$};
  \draw[edgej] (w) -- (x) node[midway, right=2pt, black] {$\nu\ (j)$};
  \draw[edgej] (v) -- (y) node[midway, left=2pt, black]  {$\nu'\ (j)$};
  \draw[edgei] (x) -- (y) node[midway, below=2pt, black] {$\mu'\ (i)$};
 
  \foreach \p in {v,w,y,x}{ \fill (\p) circle (1.7pt); }
  \node[above left =1pt] at (v) {$v=\mathrm{r}(e)$};
  \node[above right=1pt] at (w) {$w=\mathrm{s}(e)$};
  \node[below left =1pt] at (y) {$y$};
  \node[below right=1pt] at (x) {$x$};
 
  \node[align=center] at (2,2)
    {$\nu'\, e =\mu' \, \nu$\\[2pt]\footnotesize (factorisation property)};
 
  \node[draw, rounded corners, fill=black!4, align=center, font=\footnotesize,
        anchor=west] at (5.55,2.7)
    {\textbf{hypothesis}\\ $\mathrm{s}(e)$ emits $j$\\ $j\in\mathsf{E}(\mathrm{s}(e))$};
  \node[draw, rounded corners, fill=colj!12, align=center, font=\footnotesize,
        anchor=east] at (-1.55,2.7)
    {\textbf{conclusion}\\ $\therefore\ \mathrm{r}(e)$ emits $j$\\ $j\in\mathsf{E}(\mathrm{r}(e))$};
 
  \draw[->, dashed, gray, line width=0.8pt]
       (w) .. controls (3,5.5) and (1,5.5) .. (v);
  \node[font=\footnotesize, gray!40!black] at (2,5.6)
    {colour $j$ propagates backwards along $e$};
 
  \node[anchor=north, font=\footnotesize] at (2,-1.2)
    {\textcolor{coli}{\rule[0.45ex]{16pt}{1.4pt}}~~colour $i$
     \qquad
     \textcolor{colj}{\rule[0.45ex]{16pt}{1.4pt}}~~colour $j$};
 
  \node[anchor=north, font=\footnotesize, text width=11.5cm, align=center]
       at (2,-1.9)
    {\textbf{Lemma 3.4.} A colour $j$ emitted at $\mathrm{s}(e)$ is forced to
     $\mathrm{r}(e)$; together with $i$ (emitted at $\mathrm{r}(e)$ by $e$
     itself) this gives
     $\mathsf{E}(\mathrm{r}(e))\supseteq\{i\}\cup\mathsf{E}(\mathrm{s}(e))$.};
 
\end{tikzpicture}
\]
\end{lemma}

\begin{proof}
Since $e \in \rge(e)\Lambda^{\eps_i}$, we have $i\in\Em(\rge(e))$. Let
$j\in\Em(\src(e))$ with $j\ne i$. Choose $\nu\in\src(e)\Lambda^{\eps_j}$; then
$\src(e)=\rge(\nu)$, so $e\nu$ is defined and $d(e\nu)=\eps_i+\eps_j$. By the
factorisation property $e\nu=\nu'\mu'$ with $d(\nu')=\eps_j$, $d(\mu')=\eps_i$, and
$\rge(\nu')=\rge(e\nu)=\rge(e)$. Thus $\nu'\in\rge(e)\Lambda^{\eps_j}$, giving
$j\in\Em(\rge(e))$. The case $j=i$ is already covered by the first sentence.
\end{proof}

\noindent
Lemma~\ref{lem:prop} says that the emitted-colour set is monotone along edges,
toward the range: emission can only grow as one moves along the arrows in $\Lambda$.

\section{Forced concentration and the main theorem}

\begin{proposition}[Forced concentration]\label{prop:conc}
Let $\Lambda$ be a finite, singly connected, edge-level non-degenerate $k$-graph
with at least one edge. Then there is a vertex $v_0$ with
$\Em(v_0)=\{1,\dots,k\}$.
\end{proposition}

\begin{proof}
Suppose, for contradiction, that $\lvert\Em(v)\rvert\le k-1$ for every vertex $v$.

\emph{Step 1: every emitting vertex receives an edge.}
Let $e\in\Lambda^{\eps_i}$. By Lemma~\ref{lem:prop},
$\Em(\rge(e))\supseteq\{i\}\cup\bigl(\Em(\src(e))\setminus\{i\}\bigr)$, a disjoint
union, so
\[
\bigl\lvert \Em(\src(e))\setminus\{i\}\bigr\rvert
\;\le\;\lvert\Em(\rge(e))\rvert-1\;\le\;k-2 .
\]
By Lemma~\ref{lem:count} and edge-level non-degeneracy,
\[
\bigl(\Em(\src(e))\setminus\{i\}\bigr)\cup\bigl(\Rec(\rge(e))\setminus\{i\}\bigr)
=\{1,\dots,k\}\setminus\{i\},
\]
so $\lvert\Rec(\rge(e))\setminus\{i\}\rvert\ge (k-1)-(k-2)=1$. In particular
$\Rec(\rge(e))\ne\emptyset$: \emph{the range of every edge $e \in \Lambda^{\varepsilon_i}$ receives some colour.}
Since every vertex which emits a colour is the range of an edge, every emitting
vertex receives an edge.

\emph{Step 2: an infinite backward chain.}
As $\Lambda$ has an edge, pick a vertex $v_0$ that emits it. By Step~1 $v_0$ receives an edge,
so there is an edge $f_1$ with $\rge(f_1)=v_0$. Put $v_1:=\src(f_1)$; then $v_1$
emits an edge ($f_1$ witnesses it), so by Step~1 it receives an edge, giving $f_2$ with
$\rge(f_2)=v_1$, and so on. Inductively we obtain edges $f_1,f_2,\dots$ and
vertices $v_0,v_1,\dots$ with
\[
\rge(f_{n+1})=v_n=\src(f_n)\qquad(n\ge 1),
\]
so each composite $f_{n}f_{n-1}\cdots f_{m+1}$ is defined.

\emph{Step 3: contradiction.}
Since $\Lambda^0$ is finite, the sequence $(v_n)$ repeats: $v_m=v_n=:y$ for some
$m>n\ge 0$. Then (reading from right to left)
\[
\sigma:=f_{n+1} \cdots f_{m-1} f_{m}
\]
satisfies $\rge(\sigma)=\rge(f_{n+1})=v_{n+1}=y$ and $\src(\sigma)=\src(f_{m})=v_m=y$, so
$\sigma\in y\Lambda y$, while
$d(\sigma)=\sum_{\ell=n+1}^{m}\eps_{c(f_\ell)}\ne 0$. This contradicts
Lemma~\ref{lem:noloop}. Hence some vertex emits all $k$ colours.
\end{proof}

\begin{lemma}[$K_5$ at a $5$-emitting vertex]\label{lem:k5}
Let $\Lambda$ be a singly connected, locally convex $k$-graph and $v_0$ a vertex
emitting (at least) five colours, say $1,\dots,5\in\Em(v_0)$. Then $\Sk_\Lambda$
contains a subdivision of $K_5$; in particular $\Sk_\Lambda$ is non-planar.
\end{lemma}

\begin{proof}
For $\ell\in \Em (v_0) =\{1,\dots,5\}$ choose $e_\ell\in v_0\Lambda^{\eps_\ell}$ and put
$w_\ell:=\src(e_\ell)$. Fix $\ell\ne m$. Since $\rge(e_\ell)=\rge(e_m)=v_0$, local
convexity (Definition~\ref{def:lc}) yields
$e'_m\in\src(e_\ell)\Lambda^{\eps_m}$ and
$e'_\ell\in\src(e_m)\Lambda^{\eps_\ell}$ with $e_\ell e'_m=e_m e'_\ell$. Set
\[
d_{\ell m}:=\src(e_\ell e'_m)=\src(e'_m)=\src(e'_\ell).
\]
Then $\rge(e'_m)=\src(e_\ell)=w_\ell$ and $\rge(e'_\ell)=\src(e_m)=w_m$, so in
$\Sk_\Lambda$ the edges $e'_m$ and $e'_\ell$ form an undirected path
\[
w_\ell\;\text{---}\;d_{\ell m}\;\text{---}\;w_m
\]
of length two with interior vertex $d_{\ell m}$.

\emph{Distinctness of the fifteen vertices $w_\ell,d_{\ell m}$ $1 \le \ell < m \le 5$, .}
The element $e_\ell$ lies in $v_0\Lambda w_\ell$ with $d(e_\ell)=\eps_\ell$, and
$e_\ell e'_m\in v_0\Lambda d_{\ell m}$ with degree $\eps_\ell+\eps_m$. If two of the
fifteen vertices coincided, then $v_0\Lambda(\cdot)$ would contain two morphisms of
distinct degrees (the relevant degrees among $\eps_\ell$ and $\eps_\ell+\eps_m$ are
pairwise distinct, and all are nonzero), contradicting single connectedness. Hence
$w_1,\dots,w_5$ and the ten $d_{\ell m}$ are fifteen distinct vertices, and each
$d_{\ell m}$ lies on exactly one of the ten connecting paths.

Therefore $w_1,\dots,w_5$ are the branch vertices of a subdivision of $K_5$, each
of its ten edges subdivided once by a distinct $d_{\ell m}$. A graph containing a
$K_5$-subdivision is non-planar by Kuratowski's theorem \cite{Kur,Diestel}.
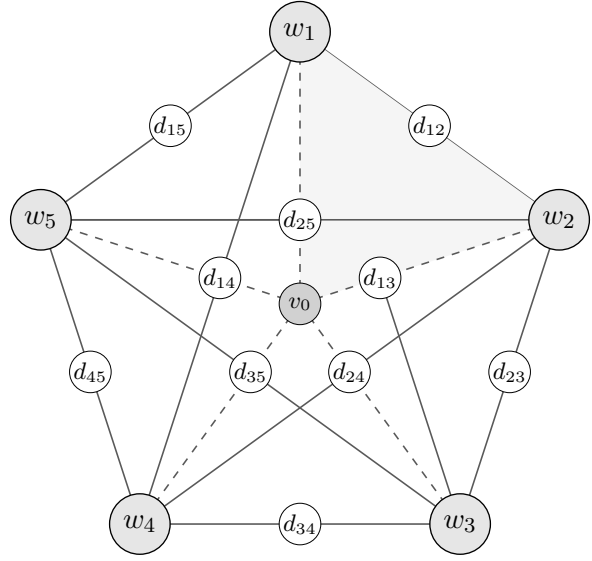
\[
\begin{tikzpicture}[
  branch/.style = {circle, draw=black, line width=0.5pt, fill=black!10,
                   minimum size=8mm, inner sep=0pt, font=\small},
  subdiv/.style = {circle, draw=black, line width=0.3pt, fill=white,
                   minimum size=5.5mm, inner sep=0pt, font=\scriptsize},
  edge/.style   = {line width=0.6pt, black!65}
]

  \def\R{3.6}
  \foreach \i in {1,...,5}{
    \coordinate (W\i) at ({90-(\i-1)*72}:\R);
  }

  \foreach \i/\j in {1/2,1/3,1/4,1/5,2/3,2/4,2/5,3/4,3/5,4/5}{
    \draw[edge] (W\i) -- (W\j);
    \coordinate (D\i\j) at ($(W\i)!0.5!(W\j)$);
  }


   Uncomment to display the apex, the five spokes, and the shaded (1,2)-square.
  
   \coordinate (V0) at (0,0);
   \fill[black!4] (V0) -- (W1) -- (D12) -- (W2) -- cycle;
   \foreach \i in {1,...,5}{ \draw[edge, dashed] (V0) -- (W\i); }
   \node[subdiv, fill=black!18] at (V0) {$v_0$};

\draw[edge] (W5) -- (W2);
    \coordinate (D25) at ($(W2)!0.5!(W5)$);
   ------------------------------------------------------------------------

  \foreach \i/\j in {1/2,1/3,1/4,1/5,2/3,2/4,2/5,3/4,3/5,4/5}{
    \node[subdiv] at (D\i\j) {$d_{\i\j}$};
  }

  \foreach \i in {1,...,5}{
    \node[branch] at (W\i) {$w_{\i}$};
  }

  \node[anchor=north, font=\small] at (0,-4.0)
    {The ten commuting squares, by corner $w_j \,d_{ij}\, w_i \, v_0$ (reading right to left):};

  \node[anchor=north, font=\footnotesize] at (0,-4.7)
  {%
    \begin{tabular}{@{}l@{\qquad}l@{}}
      $(1,2)$:\ \ $w_2 \,d_{12} \,w_1\, v_0$ & $(2,4)$:\ \ $w_4  \,d_{24}\, w_2\, v_0$\\[3pt]
      $(1,3)$:\ \ $w_3 \, d_{13}\, w_1\, v_0$ & $(2,5)$:\ \ $w_5\,d_{25}\, w_2 \, v_0$ \\[3pt]
      $(1,4)$:\ \ $w_4 \,d_{14}\, w_1 \,v_0$ & $(3,4)$:\ \ $w_4 \,d_{34}\, w_3 \, v_0$\\[3pt]
      $(1,5)$:\ \ $w_5 \,d_{15}\, w_1 \, v_0$ & $(3,5)$:\ \ $w_5 \,d_{35}\, w_3 \, v_0$\\[3pt]
      $(2,3)$:\ \ $w_3 \,d_{23}\, w_2 \, v_0$ & $(4,5)$:\ \ $w_5 \,d_{45}\, w_4 \, v_0$\\
    \end{tabular}%
  };

  \node[anchor=north, font=\scriptsize, text width=12cm, align=center] at (0,-7.3)
  {Each $(i,j)$ entry is the commuting-square relation
   $v_0\xrightarrow{\,i\,}w_i\xrightarrow{\,j\,}d_{ij}
        =v_0\xrightarrow{\,j\,}w_j\xrightarrow{\,i\,}d_{ij}$;
   the apex $v_0$ is the common range of the five colour edges.};

\end{tikzpicture}
\]
\end{proof}

\begin{proof}[Proof of Theorem~\ref{thm:main}]
Let $\Lambda$ be finite, connected, singly connected, locally convex, edge-level
non-degenerate, of rank $k\ge 5$. By Proposition~\ref{prop:conc} some vertex
$v_0$ emits all $k$ colours, in particular five of them; by Lemma~\ref{lem:k5} the $1$-skeleton
$\Sk_\Lambda$ contains a subdivision of $K_5$ and is non-planar. The final
assertion is the contrapositive: a tree of rank $\ge 5$ with these properties
cannot have planar $1$-skeleton, so a planar one has rank $\le 4$.
\end{proof}

\section{Remarks}\label{sec:remarks}

\begin{remark}[Sharpness]
The bound is sharp, and the argument  in Proposition~\ref{prop:conc} explains why. Running
Proposition~\ref{prop:conc} at rank $k=4$ still forces a vertex emitting all four
colours, but the corresponding construction in Lemma~\ref{lem:k5} produces only a
subdivision of $K_4$, which is planar. Thus the obstruction first appears at rank
$5$, where $K_5$ becomes the smallest non-planar complete graph. This is exactly
the threshold realised by the construction in \cite[Theorem~A]{Pask}, whose planar
trees occur in ranks $2,3,4$.
\end{remark}

\begin{remark}[Role of the four-colour theorem]
The theorem above is purely an obstruction: it uses only that $K_5$ is not planar
\cite{Kur}. The four-colour theorem \cite{AppelHaken} is used in \cite{Pask} for
the complementary, constructive statement that rank $4$ is always \emph{attainable}
(four colours suffice to colour the faces of a planar polyhedron). The two
ingredients of a complete ``$4$-colour-type'' theory for planar higher-rank trees
are therefore independent: achievability of $\le 4$ from the four-colour theorem,
and impossibility of $\ge 5$ from $K_5$, the latter proved here.
\end{remark}

\begin{remark}[The two readings of planarity, and $Q_4$]\label{rem:2complex}
One may instead read ``planar'' as planarity of the associated \emph{$2$-complex}
$X_\Lambda$: the CW-complex whose $0$- and $1$-cells are the vertices and edges of
$\Sk_\Lambda$ and which carries one $2$-cell glued along each commuting square (the
$2$-skeleton of the cubical complex of \cite{KPSS}). Since a complex embeds in the
sphere only if its $1$-skeleton does, planarity of $X_\Lambda$ is \emph{stronger}
than the hypothesis of Theorem~\ref{thm:main}, and the corresponding bound
\[
X_\Lambda\ \text{planar and}\ \Lambda\ \text{edge-level non-degenerate}
\ \Longrightarrow\ \operatorname{rank}\Lambda\le 4
\]
is an immediate corollary: planarity of $X_\Lambda$ forces planarity of
$\Sk_\Lambda$, whereupon Theorem~\ref{thm:main} applies.

The two readings do \emph{not} coincide, and in particular one cannot reduce the
$1$-skeleton statement to the $2$-complex one by promoting $1$-skeleton planarity to
planarity of $X_\Lambda$. The obstruction is the hypercube
$Q_4=\Omega_{4,\mathbf 1}$, regarded as a rank-$4$ $k$-graph \cite[Example~6.1]{Pask}.
It is a connected, singly connected, locally convex higher-rank tree, and it is
edge-level non-degenerate: an edge of $Q_4$ varies a single coordinate and lies on
the three two-faces obtained by additionally varying one of the other three
coordinates, so it boxes with all three remaining colours. Each edge of $Q_4$
therefore borders \emph{three} commuting squares. But any subcomplex of $S^2$ has at
most two $2$-cells along a given edge, so $X_{Q_4}$ is not a surface, hence not
planar; correspondingly $\Sk_{Q_4}$ already contains a subdivision of $K_{3,3}$ and
is non-planar \cite[Example~6.1]{Pask}. Thus the surface condition ``every edge borders
at most two squares'', which planarity of $X_\Lambda$ entails, is \emph{not} a
consequence of being a non-degenerate locally convex tree: $Q_4$ disproves it.
(As $Q_4$ has rank $4$, it is of course no counterexample to
Theorem~\ref{thm:main}, which constrains only rank $\ge 5$.) Consequently
$2$-complex planarity is a strictly stronger and non-automatic hypothesis, and the
$1$-skeleton reading adopted in Theorem~\ref{thm:main} --- that of
\cite[Theorem~A]{Pask} --- is the substantive one.
\end{remark}

\begin{remark}[Necessity of edge-level non-degeneracy]
The interaction-graph weakening of Definition~\ref{def:nd} --- requiring only that
each \emph{pair} of colours box \emph{somewhere} --- is too weak. Glue, for each
pair $\{i,j\}\subseteq\{1,\dots,5\}$, one $(i,j)$-square, identifying all of their
sink corners to a single vertex. The result is connected, singly connected,
locally convex, has planar $1$-skeleton (ten quadrilaterals meeting at a point),
is a tree, and realises every pair of colours, yet has rank $5$. It is excluded by
Definition~\ref{def:nd}, since each of its edges lies in exactly one square and so
boxes with only one other colour. This matches the description of the degenerate
high-rank examples in \cite[\S5]{Pask}.
\end{remark}

\begin{remark}[Finiteness]
Finiteness is used only in Step~2--3 of Proposition~\ref{prop:conc}: in an infinite
$k$-graph the backward chain $(v_n)$ may fail to repeat, and the conclusion can
fail. The infinite case --- relevant to the $\tilde A_2$ examples of
\cite[\S5]{Pask} --- is not addressed here.
\end{remark}

\begin{remark}[On the hypotheses]
Single connectedness is used twice, each time only through
Lemma~\ref{lem:noloop}: to close the chain in Proposition~\ref{prop:conc} and to
separate the fifteen vertices in Lemma~\ref{lem:k5}. Local convexity is used only
in Lemma~\ref{lem:k5}, to close the ten pairwise squares at $v_0$. Both hold for
the trees of \cite[Theorem~A]{Pask}. We did not use the triviality of the fundamental
group beyond what single connectedness already supplies.
\end{remark}

\begin{declaration}
I acknowledge the use of Claude to brainstorm ideas and proofread grammar. A full record of prompts and outputs is available upon request
\end{declaration}

\end{document}